\documentclass{amsart}
\usepackage{latexsym,amsfonts,amssymb,amsmath,amsthm,amscd}
\newtheorem{theorem}{Theorem}[section]

\newtheorem{proposition}[theorem]{Proposition}
\newtheorem{corollary}[theorem]{Corollary}

\theoremstyle{definition}

\newtheorem{example}[theorem]{Example}

\theoremstyle{remark}
\newtheorem{remark}[theorem]{Remark}

\numberwithin{equation}{section}

\newcommand\JM{Mierczy\'nski}
\newcommand\RR{\ensuremath{\mathbb{R}}}

\newcommand\ZZ{\ensuremath{\mathbb{Z}}}

\newcommand\PP{\ensuremath{\mathbb{P}}}

\newcommand{\OFP}{\ensuremath{(\Omega,\mathfrak{F},\PP)}}

\newcommand{\norm}[1]{\ensuremath{\lVert#1\rVert}}

\DeclareMathOperator{\trace}{tr}
\DeclareMathOperator{\col}{col}
\DeclareMathOperator{\Id}{Id}

\begin{document}
\title[Lower estimates of top exponent for cooperative systems]
{Lower estimates of top Lyapunov exponent for cooperative random
systems of linear ODEs}

\author{Janusz Mierczy\'nski} \address{Institute of Mathematics and
Computer Science, Wroc{\l}aw University of Technology, Wybrze\.ze
Wyspia\'nskiego 27, PL-50-370 Wroc{\l}aw, Poland}

\curraddr{}

\email{mierczyn@pwr.edu.pl}

\thanks{First published in \emph{Proc. Amer. Math. Soc.} \textbf{143}(3), March 2015, pp. 1127--1135, published by the American Mathematical Society. \copyright\ 2016 American Mathematical Society.
\\
Supported by project S20058/I-18.}

\commby{}

\date{} 

\begin{abstract}
For cooperative random linear systems of ordinary differential
equations a method is presented of obtaining lower estimates of the
top Lyapunov exponent.  The proofs are based on applying some
polynomial Lyapunov\nobreakdash-\hspace{0pt}like function.  Known
estimates for the dominant eigenvalue of a nonnegative matrix due to
G. Frobenius and L. Yu. Kolotilina are shown to be specializations of
our results.
\end{abstract}

\bibliographystyle{amsplain}

\maketitle

In many dynamical systems (DSs) stemming from population dynamics
{\em persistence\/} is an important concept:  mathematically
speaking, some (usually invariant, or at~least forward invariant)
subset $Y$ of the phase space $X$ should have a property that
(forward semi)trajectories of points close to $Y$ have their
$\omega$\nobreakdash-\hspace{0pt}limit sets disjoint from $Y$, or
that (forward semi)trajectories of points close to $Y$ eventually
leave some preassigned neighborhood of $Y$ (that stronger property is
usually called {\em permanence\/}).

$Y$ corresponds to some constituents of the population being absent,
therefore it has a structure of a $C^1$ submanifold, or a (forward)
invariant family of $C^1$ submanifolds.  When the dynamical system
in~question is smooth enough, its linearization leaves the tangent
bundle of $Y$ invariant, consequently one can legitimately speak of
{\em normal Lyapunov exponents\/}, that is, exponential growth rates
of those trajectories of the linearization at $Y$ which correspond to
a subbundle complementary to the tangent bundle of $Y$. It is to be
expected that positivity of such normal Lyapunov exponents should
imply persistence.

Indeed, there are a wealth of papers where the above implication is
shown, for~instance \cite{Sch}, \cite{MiSch}, \cite{GaHof} for
autonomous systems of ordinary differential equations,
\cite{MiShZhao}, \cite{MiSh-JDDE} for nonautonomous systems of
reaction\nobreakdash--\hspace{0pt}diffusion equations of Kolmogorov
type, to mention but a few.

Even in the case when the original DS is generated by an autonomous
system of ordinary differential equations (ODEs), if there are
nontrivial ergodic invariant measures on $Y$ investigation of normal
Lyapunov exponents requires considering
time\nobreakdash-\hspace{0pt}dependent systems of linear ODEs.

It is straightforward then that being able to obtain estimates
\textbf{from below} of Lyapunov exponents for
time\nobreakdash-\hspace{0pt}dependent systems of linear ODEs is of
great importance in the theory of persistence.  Sometimes, especially
in systems of ODEs encountered in animal ecology, $Y$ is a
submanifold of codimension one (or the problem can be reduced to that
case), so that one needs only to find sufficient conditions for the
least Lyapunov exponent of a
(one\nobreakdash-\hspace{0pt}dimensional) linear ODE to be positive,
which is a relatively easy task.

In applications of systems of ODEs to epidemiology, the codimension
of $Y$ may be greater than one.  Also, an important property is that
the matrix of the restriction of the linearization to a complementary
subbundle of $Y$ has off\nobreakdash-\hspace{0pt}diagonal entries
nonnegative (see~\cite{Ma-Sm}, \cite{Sal} and references contained
therein).

\smallskip
Thanks to the efforts of many generations of mathematicians, numerous
results are known on estimates of the dominant eigenvalue of a matrix
with nonnegative off\nobreakdash-\hspace{0pt}diagonal entries, see
\cite[Chapter~II]{Minc}.

One should bear in~mind, however, that for
many\nobreakdash-\hspace{0pt}dimensional systems $x' = A(t)x$ of
linear ODEs there is no straightforward relation between the real
parts of the eigenvalues of $A(t)$ and exponential growth rates of
solutions of the system (see an excellent paper \cite{JoRo}). So,
results alluded to in the previous paragraph are of limited (if~any)
value as regards finding estimates in the
time\nobreakdash-\hspace{0pt}dependent case.

\medskip
Our goal in the present paper to give easily checkable lower
estimates on maximum exponential growth rates of solutions of random
linear ODE systems whose matrices have nonnegative
off\nobreakdash-\hspace{0pt}diagonal entries.

\setcounter{section}{-1}

\section{Preliminaries}

We denote by $\RR^{N \times N}$ the set of real $N$ by $N$ matrices,
equipped with the standard structures. The entries of $A \in \RR^{N
\times N}$ are denoted by $a_{ij}$, etc.

A matrix $A \in \RR^{N \times N}$ is an {\em
ML\nobreakdash-\hspace{0pt}matrix\/} if its
off\nobreakdash-\hspace{0pt}diagonal entries are nonnegative.  For an
ML\nobreakdash-\hspace{0pt}matrix $A$ its eigenvalue with the largest
real part is in~fact real (cf., e.g., \cite[Thm.~2.1.1 on p.\
26]{BerPlem}, or \cite[Section~2.3]{Sen}), will be called the {\em
dominant eigenvalue\/} of $A$, and denoted by $d(A)$.  If $A$ has all
entries nonnegative then $d(A)$ is equal to the spectral radius of
$A$.

\smallskip
$\langle \cdot, \cdot \rangle$ denotes the standard (Euclidean) inner
product in $\RR^N$.  $\norm{\cdot}$ will stand, depending on the
context, either for the standard Euclidean norm on $\RR^N$, or for
the operator norm on $\mathcal{L}(\RR^N)$, the vector space of linear
maps from $\RR^N$ into $\RR^N$, induced by the standard Euclidean
norm.

\smallskip
We denote
\begin{equation*}
\begin{aligned}
\RR^{N}_{+}& :=
\{\, x = \col{(x_1, \dots, x_n)} : x_i \ge 0 \ , 1 \le i \le n \,\}, \\
\RR^{N}_{++} & := \{\, x = \col{(x_1, \dots, x_n)} : x_i > 0 \ , 1
\le i \le n \,\}.
\end{aligned}
\end{equation*}

Our method consists in finding a sort of Lyapunov function, more
precisely, a homogeneous polynomial $V$ of the coordinates $x_1,
\dots, x_N$ of $x$, taking positive values on $\RR^{N}_{++}$, with
the property that
\begin{itemize}
\item
there exists $a > 0$ such that
\begin{equation*}
\norm{x} \ge a V(x) \quad \text{for all } x \in \RR^{N}_{++};
\end{equation*}
\item
there is a function $b(\cdot)$ defined on the set of
ML\nobreakdash-\hspace{0pt}matrices, expressed as an elementary
function of the entries (in our examples, by~means only of
arithmetic operations, taking square roots and/or maxima/minima),
such that for any ML\nobreakdash-\hspace{0pt}matrix $A$ the
inequality
\begin{equation*}
\langle \nabla V(x), Ax \rangle \ge b(A) \, V(x)
\end{equation*}
holds for all $x \in \RR^{N}_{++}$.
\end{itemize}

The paper is organized as follows.

In Section~\ref{section:Kolotilina} the Lyapunov function $V$ is
defined as $V(x) = x_1 \cdot \ldots \cdot x_{N}$. The estimates
obtained (Theorem~\ref{thm:ODEs}) are called {\em Kolotilina-type\/}
estimates, since their counterparts for nonnegative matrices first
appeared in L. Yu. Kolotilina's paper~\cite{Kolo}.

In Section~\ref{section:Frobenius} the polynomial $V(x) = x_1 +
\ldots + x_N$ is used. For time\nobreakdash-\hspace{0pt}independent
matrices the results obtained there
(Theorems~\ref{thm:ODEs-Frobenius-columns}
and~\ref{thm:ODEs-Frobenius-rows}) specialize to
well\nobreakdash-\hspace{0pt}known Frobenius estimates (see, e.g.,
\cite[3.1.1]{MarMinc}, or \cite[Chapter~II]{Minc}).

Section~\ref{section:extensions} gives possible extensions of the
results.  Section~\ref{section:concluding} contains concluding
remarks

\section{Kolotilina-type estimates}
\label{section:Kolotilina}

\subsection{Basic estimate}
\label{subsection:basic-estimate}
Throughout the present subsection we assume that $A \colon (\alpha,
\beta) \to \RR^{N \times N}$, where $-\infty \le \alpha < \beta \le
\infty$, with locally integrable entries.

Consider a system of linear ODEs
\begin{equation}
\label{ODE-system-nonautonomous}
x' = A(t) x, \quad t \in (\alpha, \beta).
\end{equation}
For any $t_0 \in (\alpha, \beta)$ and any $x_0 \in \RR^N$ there
exists a unique solution $[\, (\alpha, \beta) \ni t \mapsto x(t; t_0,
x_0) \in \RR^N \,]$ of \eqref{ODE-system-nonautonomous} satisfying
the initial condition $x(t_0) = x_0$. The solution is understood in
the Carath\'eodory sense: The function $[\, t \mapsto x(t; t_0, x_0)
\,]$ is absolutely continuous on any compact subinterval of $(\alpha,
\beta)$, $x(t_0; t_0, x_0) = x_0$, and
\eqref{ODE-system-nonautonomous} is satisfied for
Lebesgue\nobreakdash-\hspace{0pt}a.e.\ $t \in (\alpha, \beta)$.

System~\eqref{ODE-system-nonautonomous} is said to be {\em
cooperative\/} if for each $t \in (\alpha, \beta)$ the matrix $A(t)$
is an ML\nobreakdash-\hspace{0pt}matrix.

\begin{proposition}
\label{prop:basic-estimate}
Assume that \eqref{ODE-system-nonautonomous} is cooperative.  Then
for each $t_0 \in (\alpha, \beta)$ and each $x_0 \in \RR^{N}_{++}$,
if we denote $x(t; t_0, x_0) = \col(x_1(t), \dots, x_N(t)))$, the
inequality
\begin{equation}
\label{eq:basic-estimate}
\prod_{i} x_{i}(t) \ge \exp{\biggl( \int\limits_{t_0}^{t} \Bigl(
\trace{A(\tau)} + 2 \sum_{j < k} \sqrt{a_{jk}(\tau) \,
a_{kj}(\tau)}\: \Bigr) \, d\tau \biggr)} \prod_{i} x_{i}(t_0)
\end{equation}
holds for all $t \in (t_0, \beta)$.
\end{proposition}
\begin{proof}
Since the functions $x_i(\cdot)$, $1 \le i \le N$, are a.e.\
differentiable, we have that a.e.\ on $(t_0, \beta)$ there holds
\begin{equation*}
\begin{aligned}
\Bigl( \prod_{i} x_i \Bigr)' = {} & \sum_{i} \Bigl( x'_{i} \prod_{j
\ne i} x_j  \Bigr) = \sum_{i} \biggl( \Bigl( \sum_{k} a_{ik}(t) x_{k}
\Bigr) \prod_{j \ne i} x_j  \biggr) \\
= {} & \sum_{i} \Bigl( a_{ii}(t) \, x_{i} \prod_{j \ne i} x_j \Bigr)
+ \sum_{i} \biggl( \Bigl( \sum_{k \ne i} a_{ik}(t) x_{k} \Bigr)
\prod_{j \ne i} x_j  \biggr).
\end{aligned}
\end{equation*}
The first term is equal to
\begin{equation*}
\trace{A(t)} \prod_{i} x_i.
\end{equation*}
Regarding the second term, we have
\begin{equation*}
\sum_{i} \biggl( \Bigl( \sum_{k \ne i} a_{ik}(t) x_{k} \Bigr)
\prod_{j \ne i} x_j  \biggr) = \sum_{j < k} \Bigl( \bigl( a_{jk}(t)
(x_k)^2 + a_{kj}(t) (x_j)^2 \bigr) \prod_{\substack{i \ne j \\ i \ne
k}} x_i \Bigr).
\end{equation*}
But, as
\begin{equation*}
a_{jk}(t) (x_k)^2 + a_{kj}(t) (x_j)^2 \ge 2 \sqrt{a_{jk}(t) \,
a_{kj}(t)} \, x_{j} x_{k}
\end{equation*}
for any $j < k$, we have
\begin{equation}
\label{eq:0}
\Bigl( \sum_{i} \ln{x_i} \Bigr)' \ge \trace{A(t)} + 2 \sum_{j < k}
\sqrt{a_{jk}(t) \, a_{kj}(t)}
\end{equation}
a.e.\ on $(t_0, \beta)$.  Keeping in~mind that the functions
$x_i(\cdot)$ are absolutely continuous on compact subintervals, we
can integrate the above inequality to get the desired result.
\end{proof}

\subsection{Monotone linear random dynamical systems}
\label{subsection:monotone}
In this subsection we present an application of
Proposition~\ref{prop:basic-estimate} to monotone linear random
dynamical systems generated by random cooperative systems of linear
ODEs.  We start by sketching briefly the construction of such a DS.

\smallskip
Let $\OFP$ be a probability space: $\Omega$ is a set, $\mathfrak{F}$
is a $\sigma$\nobreakdash-\hspace{0pt}algebra of subsets of $\Omega$
and $\PP$ is a probability measure defined for all $F \in
\mathfrak{F}$. We assume in addition that the measure $\PP$ is
complete.

For a metric space $S$, $\mathfrak{B}(S)$ denotes the
$\sigma$\nobreakdash-\hspace{0pt}algebra of Borel subsets of $S$

\smallskip
By a {\em measure\nobreakdash-\hspace{0pt}preserving dynamical
system\/} on $\OFP$ we understand a $(\mathfrak{B}(\RR) \otimes
\mathfrak{F}, \mathfrak{F})$\nobreakdash-\hspace{0pt}measurable
mapping $\theta \colon \RR \times \Omega \to \Omega$ satisfying the
following (where we write, for $t \in \RR$, $\theta_t(\cdot) =
\theta(t, \cdot)$):
\begin{itemize}
\item
$\theta_0 = \Id_{\Omega}$,
\item
$\theta_{s + t} = \theta_{t} \circ \theta_{s}$, for any $s, t \in
\RR$,
\item
$\PP(\theta_{t}(F)) = \PP(F)$, for any $t \in \RR$ and any $F \in
\mathfrak{F}$.
\end{itemize}
We denote a measure\nobreakdash-\hspace{0pt}preserving DS as
$(\OFP,(\theta_t)_{t\in\RR})$.  We usually write $\omega \cdot t$
instead~of $\theta_{t}(\omega)$.

We always assume that $(\OFP,(\theta_t)_{t\in\RR})$ is {\em
ergodic\/}: For any invariant $F \in \mathfrak{F}$, either $\PP(F) =
0$ or $\PP(F) = 1$.

The first assumption we make on $A$ is:
\begin{enumerate}
\item[\textbf{(A0)}]
{\em $A \colon \Omega \to \RR^{N \times N}$ belongs to $L_1(\OFP,
\RR^{N \times N})$.}
\end{enumerate}

Under (A0) it can be shown (see \cite[Ex.~2.2.8]{Arn}) that for any
$\omega \in \Omega$ and any $x_0 \in \RR^N$ there exists a unique
(Carath\'eodory)  solution $[\, \RR \ni t \mapsto x(t; \omega, x_0)
\in \RR^N \,]$ of the {\em random linear system of ODEs\/}
\begin{equation}
\label{ODE-system}
x' = A(\omega \cdot t) x
\end{equation}
satisfying the initial condition $x(0) = x_0$. Moreover, the mapping
\begin{equation*}
[\, \Omega \times \RR \times \RR^N \ni (\omega, t, x_0) \mapsto x(t;
\omega, x_0) \in \RR^N \,]
\end{equation*}
is $(\mathfrak{F} \otimes \mathfrak{B}(\RR) \otimes
\mathfrak{B}(\RR^N),
\mathfrak{B}(\RR^N))$\nobreakdash-\hspace{0pt}measurable.

For $\omega \in \Omega$ and $t \in \RR$ define $U_{\omega}(t) \in
\mathcal{L}(\RR^N)$ by
\begin{equation*}
U_{\omega}(t)x_0 := x(t; \omega, x_0), \quad x_0 \in \RR^N.
\end{equation*}
The family of linear operators $\{ U_{\omega}(t)\}_{\omega \in
\Omega, t \in \RR}$ has the following properties:
\begin{itemize}
\item
$U_{\omega}(0) = \Id_{\RR^N}$, for any $\omega \in \Omega$.
\item
$U_{\omega}(t + s) = U_{\omega \cdot s}(t) \circ U_{\omega}(s)$,
for any $\omega \in \Omega$, $s, t \in \RR$.
\end{itemize}
We call $\{ U_{\omega}(t)\}_{\omega \in \Omega, t \in \RR}$ the {\em
linear random dynamical system\/} generated by \eqref{ODE-system}.

%It follows also from (A0) that the function
%\begin{equation*}
%[\, \Omega \ni \omega \mapsto \sup\limits_{-1 \le t \le 1}
%\ln{\norm{X_{\omega}(t)}} \in \RR \, ]
%\end{equation*}
%belongs to $L_1(\OFP)$ (see \cite[Ex.~3.4.15]{Arn}).

\begin{proposition}
\label{prop:top-Lyapunov}
Assume \textup{(A0)}. Then there exist:
\begin{itemize}
\item
an invariant $\Omega_1 \subset \Omega$ with $\PP(\Omega_1) = 1$,
and
\item
a real number $\lambda$,
\end{itemize}
with the property that
\begin{equation*}
\lim\limits_{t \to \infty} \frac{\ln{\norm{U_{\omega}(t)}}}{t} =
\lambda
\end{equation*}
for each $\omega \in \Omega_1$.
\end{proposition}
Such a $\lambda$ will be referred to as the {\em top Lyapunov
exponent\/} for \eqref{ODE-system}.
\begin{proof}
See, e.g., \cite[Thm.~3.3.10]{Arn}.
\end{proof}

\medskip
The next assumption we make is:
\begin{enumerate}
\item[\textbf{(A1)}]
{\em For each $\omega \in \Omega$ the matrix $A(\omega)$ is an
ML\nobreakdash-\hspace{0pt}matrix.}
\end{enumerate}
We will call~\eqref{ODE-system} for which (A1) is satisfied a {\em
cooperative random linear system of ODEs\/}.

An important property of the linear random DS generated
by~\eqref{ODE-system} satisfying (A0) and (A1) is that it is {\em
monotone\/}: $U_{\omega}(t) \RR^{N}_{+} \subset \RR^{N}_{+}$ and
$U_{\omega}(t) \RR^{N}_{++} \subset \RR^{N}_{++}$, for all $\omega
\in \Omega$ and $t \ge 0$ (for a proof see, e.g.,
\cite[Lemma~5.2.1]{Chueshov}, or \cite{Mi-arXiv}).

We are ready now to formulate the main theorem of the paper.
\begin{theorem}[Kolotilina-type estimate]
\label{thm:ODEs}
Under \textup{(A0)} and \textup{(A1)},
\begin{equation*}
\lambda \ge \frac{1}{N} \int\limits_{\Omega} \Bigl( \trace{A(\omega)}
+ 2 \sum_{j < k} \sqrt{a_{jk}(\omega) \, a_{kj}(\omega)}\: \Bigr) \,
d\mathbb{P}(\omega).
\end{equation*}
\end{theorem}
\begin{proof}
Since $0 \le \sqrt{a_{jk} \, a_{kj}} \le \frac{1}{2} (a_{jk} +
a_{kj})$, it follows from (A0) that the functions
$\sqrt{a_{jk}(\cdot) \, a_{kj}(\cdot)}$ are in $L_1(\OFP)$, for any
$j < k$.  An application of the Birkhoff ergodic theorem (see, e.g.,
\cite[Appendix~A.1]{Arn}) yields the existence of $\tilde{\Omega}
\subset \Omega$, $\PP(\tilde{\Omega}) = 1$, such that for each
$\omega \in \tilde{\Omega}$ there holds
\begin{equation}
\label{eq:1}
\begin{aligned}
\lim_{t \to \infty} \frac{1}{t} \int\limits_{0}^{t} \Bigl(
\trace{A(\omega \cdot \tau)} {} & + 2 \sum_{j < k}
\sqrt{a_{jk}(\omega \cdot \tau)
\, a_{kj}(\omega \cdot \tau)}\: \Bigr) \, d\tau \\
{} & = \int\limits_{\Omega} \Bigl( \trace{A(\cdot)} + 2 \sum_{j < k}
\sqrt{a_{jk}(\cdot) \, a_{kj}(\cdot)}\: \Bigr) \, d\mathbb{P}(\cdot).
\end{aligned}
\end{equation}
Fix $\omega \in \tilde{\Omega} \cap \Omega_1$, where $\Omega_1$ is as
in Proposition~\ref{prop:top-Lyapunov}.  By taking $t_0 = 0$, $x_0 =
\col(1, \dots, 1)$ and $x(t) = x(t; t_0, x_0) = x(t; \omega, x_0)$ in
Proposition~\ref{prop:basic-estimate} we obtain, after
straightforward calculation, that
\begin{equation}
\label{eq:2}
\begin{aligned}
&
\ln{\norm{x(t)}} \ge \max_{i} \ln{x_i(t)} \ge \frac{1}{N} \sum_{i}
\ln{x_i(t)}
\\
\ge {} &
\frac{1}{N} \int\limits_{0}^{t} \Bigl( \trace{A(\omega \cdot \tau)} +
2 \sum_{j < k} \sqrt{a_{jk}(\omega \cdot \tau) \, a_{kj}(\omega \cdot
\tau)}\: \Bigr) \, d\tau
\end{aligned}
\end{equation}
for each $t > 0$.

Further,
\begin{equation}
\label{eq:3}
\ln{\norm{U_{\omega}(t)}} \ge \ln{\norm{x(t)}} - \tfrac{1}{2} \ln{N},
\quad t > 0.
\end{equation}
Gathering Proposition~\ref{prop:top-Lyapunov}, \eqref{eq:1},
\eqref{eq:2} and \eqref{eq:3} concludes the proof.
\end{proof}

It is good to pause here to explain the reason why we have chosen to
put our results in the setting of random DSs.  Indeed, this setting
covers a~lot of cases.  For~instance, consider a nonautonomous linear
ODE system
\begin{equation}
\label{eq:ODE-nonautonomous}
x' = \tilde{A}(t) x
\end{equation}
for which the function  $\tilde{A} \colon (-\infty, \infty) \to
\RR^{N \times N}$ (or $\tilde{A} \colon [0, \infty) \to \RR^{N \times
N}$) is bounded and continuous.  \eqref{eq:ODE-nonautonomous} can be
embedded into a family
\begin{equation*}
x' = A(\omega \cdot t) x
\end{equation*}
parameterized by $\omega \in \Omega$, where a compact metrizable
space $\Omega$ is the closure, in a suitable topology, of some set
whose members are time\nobreakdash-\hspace{0pt}translates of (some
extension of) $\tilde{A}$.  $\Omega$ is considered with the
translation flow on it.  For details of a similar construction the
interested reader is referred to \cite[pp.\ 81--82]{MiSh-Fields}.
Now, by the theory presented in~\cite{JoPaSe}, extreme values of the
exponential growth rates of solutions of~\eqref{eq:ODE-nonautonomous}
can be expressed as Lyapunov exponents for some invariant ergodic
measures for the translation flow on $\Omega$.

\subsection{Examples}
\label{subsection:examples}
\begin{example}
Assume that $A \colon \RR \to \RR^{N \times N}$ is a continuous
$T$\nobreakdash-\hspace{0pt}periodic matrix function such that $A(t)$
is an ML\nobreakdash-\hspace{0pt}matrix for all $t \in \RR$.

Then $\Omega = \RR/\{\, kT: k \in \ZZ \,\}$, $\PP$ is the normalized
Lebesgue measure on $\Omega$, and $\omega \cdot t = \omega + t$ for
any $\omega \in \Omega$, $t \in \RR$ (addition is considered
modulo~$T$).

Fix some $\omega \in \Omega$ ($\omega = 0$, say) and take the {\em
Poincar\'e map\/} $U_{0}(T)$.  Let $P$ denote the matrix of
$U_{0}(T)$ in the standard basis. $P$ is nonsingular and has, by
monotonicity, all entries nonnegative, so its (positive) spectral
radius equals its dominant eigenvalue $d(P)$.  It follows via Floquet
theory (see, e.g., \cite[Section~2.2]{Farkas}) that the top Lyapunov
exponent of the linear DS generated by the
time\nobreakdash-\hspace{0pt}periodic cooperative linear ODE system
$x' = A(t) x$ is equal to $(1/T) \ln{d(P)}$.  The estimate in
Theorem~\ref{thm:ODEs} takes the form:
\begin{equation*}
\lambda \ge \frac{1}{NT} \int\limits_{0}^{T} \Bigl( \trace{A(t)} + 2
\sum_{j < k} \sqrt{a_{jk}(t) \, a_{kj}(t)}\: \Bigr) \, dt.
\end{equation*}
\end{example}

\begin{example}
 Assume that $A \in \RR^{N \times N}$ is an
ML\nobreakdash-\hspace{0pt}matrix, and consider the autonomous system
of cooperative linear ODEs $x' = Ax$.

Then $\Omega$ is a singleton, $U(t) = e^{tA}$ for each $t \in \RR$,
so we recover the following estimate due to L. Yu. Kolotilina
(see~\cite[Corollary~3]{Kolo}, cf.\ also~\cite{Schwenk}):
\begin{corollary}[Kolotilina's estimate]
\label{cor:basic}
Let $A \in \RR^{N \times N}$ be an ML\nobreakdash-\hspace{0pt}matrix.
Then its dominant eigenvalue $d(A)$ satisfies
\begin{equation*}
d(A) \ge \frac{1}{N} \Bigl( \trace{A} + 2 \sum_{j < k}
\sqrt{a_{jk} \, a_{kj}} \Bigr).
\end{equation*}
\end{corollary}
\end{example}

\section{Frobenius-type estimates}
\label{section:Frobenius}
In the present section we apply a different Lyapunov function.  This
approach is especially useful in the case when matrices have a~lot of
zero off\nobreakdash-\hspace{0pt}diagonal entries.  We outline only
sketches of proofs.

\begin{proposition}
\label{prop:basic-estimate-Frobenius}
In the assumptions of~Subsection~\ref{subsection:basic-estimate}, for
each $t_0 \in (\alpha, \beta)$ and each $x_0 \in \RR^{N}_{++}$, if we
denote $x(t; t_0, x_0) = \col(x_1(t), \dots, x_N(t)))$, the
inequality
\begin{equation}
\label{eq:basic-estimate-Frobenius}
\sum_{i} x_{i}(t) \ge \exp{\Bigl( \min\limits_{i} \sum_{j} a_{ij}(t)
\Bigr)}  \sum_{i} x_{i}(t_0)
\end{equation}
holds for all $t \in (t_0, \beta)$.
\end{proposition}
\begin{proof}
We observe that
\begin{equation*}
\sum_{i} x'_{i} = \sum_{i} \sum_{j} a_{ij}(t) x_{j} \ge \Bigl(
\min\limits_{i} \sum_{j} a_{ij}(t) \Bigr) \sum_{i} x_{i}
\end{equation*}
holds a.e.\ on $(t_0, \beta)$, and modify the proof of
Proposition~\ref{prop:basic-estimate} accordingly.
\end{proof}

\begin{theorem}[Frobenius-type estimate, column version]
\label{thm:ODEs-Frobenius-columns}
Under \textup{(A0)} and \textup{(A1)},
\begin{equation*}
\lambda \ge \int\limits_{\Omega} \Bigl( \min\limits_{i} \sum_{j}
a_{ij}(\omega) \Bigr) \, d\mathbb{P}(\omega).
\end{equation*}
\end{theorem}
\begin{proof}[Indication of proof]
We copy the proof of Theorem~\ref{thm:ODEs}, applying the fact that
the sum of the coordinates of a vector $x$ in $\RR^N_{++}$ equals the
$\ell_1$\nobreakdash-\hspace{0pt}norm of $x$.
\end{proof}

By passing to the dual system and noticing that it is generated by a
random linear system of ODEs with matrix transposes
(see~\cite[Ch.~5]{Arn}), we see that the following holds true:
\begin{theorem}[Frobenius-type estimate, row version]
\label{thm:ODEs-Frobenius-rows}
Under \textup{(A0)} and \textup{(A1)},
\begin{equation*}
\lambda \ge \int\limits_{\Omega} \Bigl( \min\limits_{j} \sum_{i}
a_{ij}(\omega) \Bigr) \, d\mathbb{P}(\omega).
\end{equation*}
\end{theorem}
Specializing to a time\nobreakdash-\hspace{0pt}independent
ML\nobreakdash-\hspace{0pt}matrix $A$ we recover the
well\nobreakdash-\hspace{0pt}known (lower) Frobenius estimates (see,
e.g., \cite{MarMinc}):
\begin{equation*}
d(A) \ge \min\limits_{i} \sum_{j} a_{ij}, \quad d(A) \ge
\min\limits_{j} \sum_{i} a_{ij}.
\end{equation*}

\section{Extensions}
\label{section:extensions}
When some diagonal entries are large negative, the contribution from
them can worsen the estimates considerably.  That can be overcome by
considering other polynomials, for instance $V(x) = x_{i} x_{j}$ for
suitably chosen indices $i \ne j$ such that $a_{ii}(\cdot)$ and
$a_{jj}(\cdot)$ are not too large negative.

Indeed, when starting from $x(t_0) = \col{(1, \dots, 1)}$ we obtain
\begin{equation*}
\begin{aligned}
(x_{i} x_{j})' = x'_{i} x_{j} + x_{i} x'_{j} & \ge a_{ji}(t)
(x_{i})^2 + a_{ij}(t) (x_{j})^2 + (a_{ii}(t) + a_{jj}(t)) x_i x_j
\\
&
\ge \Bigl( a_{ii}(t) + a_{jj}(t) + 2 \sqrt{a_{ij}(t) a_{ji}(t)}\,
\Bigr) x_{i} x_{j},
\end{aligned}
\end{equation*}
and remembering that
\begin{equation*}
\ln{\norm{x(t)}} \ge \max(\ln{x_i(t)}, \ln{x_j(t)}) \ge \frac{1}{2}
(\ln{x_i(t)} + \ln{x_j(t)})
\end{equation*}
we finally obtain, proceeding along the lines of proof of
Theorem~\ref{thm:ODEs}, that
\begin{equation*}
\lambda \ge \int_{\Omega} \Bigl( \frac{1}{2}(a_{ii}(\omega) +
a_{jj}(\omega)) + \sqrt{a_{ij}(\omega) \, a_{ji}(\omega)}\: \Bigr) \,
d\mathbb{P}(\omega).
\end{equation*}
For a time\nobreakdash-\hspace{0pt}independent specialization see
\cite[Remark~3 on p.~142]{Kolo}.

\section{Concluding remarks}
\label{section:concluding}
\begin{remark}
Theorems~\ref{thm:ODEs}, \ref{thm:ODEs-Frobenius-columns},
\ref{thm:ODEs-Frobenius-rows} are of much importance in the case when
the so\nobreakdash-\hspace{0pt}called {\em generalized exponential
separation\/} holds: as shown in~\cite[Thms.~4.1(3)
and~2.4(3)]{MiShPart2}, under additional assumptions on matrices
$A(\cdot)$ (too complicated to be written here, specializing to
irreducibility in the autonomous case), for $\PP$-a.e.\ $\omega \in
\Omega$ the top Lyapunov exponent $\lambda$ is equal to the limit
\begin{equation*}
\lim\limits_{t \to \infty} \frac{\ln{\norm{U_{\omega}(t) x_0}}}{t}
\end{equation*}
for \textbf{each} nonzero $x_0 \in \RR^{N}_{+}$ (in that context,
$\lambda$ is called the {\em principal Lyapunov exponent\/} for
\eqref{ODE-system}).
\end{remark}

\begin{remark}
In the discrete\nobreakdash-\hspace{0pt}time case analogs of
cooperative systems of the form~\eqref{ODE-system} are random systems
of matrices having \textbf{all} entries nonnegative.  Such systems
occur, for~instance, in Leslie population models.

In~\cite{BenSch} the function $V(x) = x_1 \cdot \ldots \cdot x_{N}$
was used to obtain lower estimates of the principal characteristic
exponent in a random Leslie matrix model. Those estimates were
expressed in terms of the permanents of matrices.
\end{remark}

\end{document}